\documentclass[12pt, reqno, twoside, letterpaper]{amsart}

\usepackage{amsmath,amssymb,amsbsy,amsfonts,amsthm,latexsym,
	amsopn,amstext,amsxtra,euscript,amscd,stmaryrd,mathrsfs,
	cite,array}

\numberwithin{equation}{section}

\usepackage{color}
\usepackage[usenames,dvipsnames,svgnames,table]{xcolor}

\usepackage{color}
\usepackage{hyperref}
\hypersetup{
	pdftoolbar=true,
	pdfmenubar=false,
	pdffitwindow=false,
	pdfstartview={FitH},
	pdftitle={A generalization of Piatetski-Shapiro sequences}, 
	pdfauthor={Jinjiang Li, Jinyun Qi, Min Zhang}, 
	pdfcreator={Jinjiang Li, Jinyun Qi, Min Zhang}, 
	pdfproducer={Jinjiang Li, Jinyun Qi, Min Zhang}, 
	pdfnewwindow=true,
	colorlinks=true,
	linkcolor={red},
	citecolor={blue},
	filecolor={black},
	urlcolor={black},
}

\usepackage[text={144mm,233mm},centering]{geometry} 
\usepackage{setspace}
\usepackage{amsthm}
\newtheoremstyle{customthm}
{1em}                    
{1em}                    
{\itshape}               
{}                       
{\scshape}               
{.}                      
{5pt plus 1pt minus 1pt} 
{}                       

\newtheoremstyle{customrem}
{1em}                    
{1em}                    
{}                       
{}                       
{\scshape}               
{.}                      
{5pt plus 1pt minus 1pt} 
{}                       

\theoremstyle{customthm}

\newtheorem{X}{X}[section]
\newtheorem{theorem}[X]{Theorem}

\newtheorem{lemma}[X]{Lemma}
\newtheorem{corollary}[X]{Corollary}

\theoremstyle{customrem}

\usepackage{etoolbox}
\AtEndEnvironment{remark}{\null\hfill\qedsymbol}

\AtEndEnvironment{definition}{\null\hfill\qedsymbol}

\renewcommand{\ge}{\ensuremath{\geqslant}}

\renewcommand{\pod}[1]{\mathchoice
	{\allowbreak \if@display \mkern 5mu\else \mkern 5mu\fi (#1)}
	{\allowbreak \if@display \mkern 5mu\else \mkern 5mu\fi (#1)}
	{\mkern4mu(#1)}
	{\mkern4mu(#1)}
}

\DeclareSymbolFont{EUEX}{U}{euex}{m}{n}

\DeclareSymbolFont{euexlargesymbols}{U}{euex}{m}{n}
\DeclareMathSymbol{\intop}{\mathop}{euexlargesymbols}{"52}
\def\int{\intop\nolimits}

\DeclareSymbolFont{euexsymbols}     {U}{euex}{m}{n}
\DeclareMathSymbol{\smallint}{\mathop}{euexsymbols}{"52}

\usepackage{ifthen}
\usepackage{xifthen}

\def\sums{
	\@ifnextchar[
	{\sums@i}
	{\ensuremath{\sum}}
}
\def\sums@i[#1]{
	\@ifnextchar[
	{\sums@ii{#1}}
	{\ensuremath{\sum_{#1}}}
}
\def\sums@ii#1[#2]{
	\@ifnextchar[
	{\sums@iii{#1}{#2}}
	{\ensuremath{\sum_{\substack{#1 \\ #2}}}}
}
\def\sums@iii#1#2[#3]{
	\@ifnextchar[
	{\sums@iv{#1}{#2}{#3}}
	{\ensuremath{\sum_{\substack{#1 \\ #2 \\ #3}}}}
}
\def\sums@iv#1#2#3[#4]{
	\@ifnextchar[
	{\sums@v{#1}{#2}{#3}{#4}}
	{\ensuremath{\sum_{\substack{#1 \\ #2 \\ #3 \\ #4}}}}
}
\def\sums@v#1#2#3#4[#5]{
	{\ensuremath{\sum_{\substack{#1 \\ #2 \\ #3 \\ #4 \\ #5}}}}
}

\def\sumss[#1]{
	\@ifnextchar[
	{\sumss@i[#1]}
	{
		\ifthenelse{\isempty{#1}}
		{\ensuremath{\sum}}
		{
			\ifthenelse{\equal{#1}{'}}
			{\ensuremath{\sideset{}{^{\prime}}{\sum}}}
			{\ensuremath{\sideset{}{^{#1}}{\sum}}}
		}
	}
}
\def\sumss@i[#1][#2]{
	\@ifnextchar[
	{\sumss@ii[#1]{#2}}
	{
		\ifthenelse{\isempty{#1}}
		{\ensuremath{\sum_{#2}}}
		{
			\ifthenelse{\equal{#1}{'}}
			{\ensuremath{\sideset{}{^{\prime}}{\sum}_{#2}}}
			{\ensuremath{\sideset{}{^{#1}}{\sum}_{#2}}}
		}
	}
}
\def\sumss@ii[#1]#2[#3]{
	\@ifnextchar[
	{\sumss@iii[#1]{#2}{#3}}
	{
		\ifthenelse{\isempty{#1}}
		{\ensuremath{\sum_{\substack{#2 \\ #3}}}}
		{
			\ifthenelse{\equal{#1}{'}}
			{\ensuremath{\sideset{}{^{\prime}}{\sum}_{\substack{#2 \\ #3}}}}
			{\ensuremath{\sideset{}{^{#1}}{\sum}_{\substack{#2 \\ #3}}}}
		}
	}
}
\def\sumss@iii[#1]#2#3[#4]{
	\@ifnextchar[
	{\sumss@iv[#1]{#2}{#3}{#4}}
	{
		\ifthenelse{\isempty{#1}}
		{\ensuremath{\sum_{\substack{#2 \\ #3 \\ #4}}}}
		{
			\ifthenelse{\equal{#1}{'}}
			{\ensuremath{\sideset{}{^{\prime}}{\sum}_{\substack{#2 \\ #3 \\ #4}}}}
			{\ensuremath{\sideset{}{^{#1}}{\sum}_{\substack{#2 \\ #3 \\ #4}}}}
		}
	}
}
\def\sumss@iv[#1]#2#3#4[#5]{
	\@ifnextchar[
	{\sumss@v[#1]{#2}{#3}{#4}{#5}}
	{
		\ifthenelse{\isempty{#1}}
		{\ensuremath{\sum_{\substack{#2 \\ #3 \\ #4 \\ #5}}}}
		{
			\ifthenelse{\equal{#1}{'}}
			{\ensuremath{\sideset{}{^{\prime}}{\sum}_{\substack{#2 \\ #3 \\ #4 \\ #5}}}}
			{\ensuremath{\sideset{}{^{#1}}{\sum}_{\substack{#2 \\ #3 \\ #4 \\ #5}}}}
		}
	}
}
\def\sumss@v[#1]#2#3#4#5[#6]{
	{\ifthenelse{\isempty{#1}}
		{\ensuremath{\sum_{\substack{#2 \\ #3 \\ #4 \\ #5 \\ #6 }}}}
		{
			\ifthenelse{\equal{#1}{'}}
			{\ensuremath{\sideset{}{^{\prime}}{\sum}_{\substack{#2 \\ #3 \\ #4 \\ #5 \\ #6 }}}}
			{\ensuremath{\sideset{}{^{#1}}{\sum}_{\substack{#2 \\ #3 \\ #4 \\ #5 \\ #6 }}}}
		}
	}
}

\def\sumsstxt[#1]{
	\@ifnextchar[
	{\sumsstxt@i[#1]}
	{
		\ifthenelse{\isempty{#1}}
		{\ensuremath{\textstyle\sum}}
		{
			\ifthenelse{\equal{#1}{'}}
			{\ensuremath{\sideset{}{^{\prime}}{\textstyle\sum}}}
			{\ensuremath{\sideset{}{^{#1}}{\textstyle\sum}}}
		}
	}
}
\def\sumsstxt@i[#1][#2]{
	\@ifnextchar[
	{\sumsstxt@ii[#1]{#2}}
	{
		\ifthenelse{\isempty{#1}}
		{\ensuremath{\textstyle\sum_{#2}}}
		{
			\ifthenelse{\equal{#1}{'}}
			{\ensuremath{\sideset{}{^{\prime}}{\textstyle\sum}_{#2}}}
			{\ensuremath{\sideset{}{^{#1}}{\textstyle\sum}_{#2}}}
		}
	}
}
\def\sumsstxt@ii[#1]#2[#3]{
	\@ifnextchar[
	{\sumsstxt@iii[#1]{#2}{#3}}
	{
		\ifthenelse{\isempty{#1}}
		{\ensuremath{\textstyle\sum_{\substack{#2 \\ #3}}}}
		{
			\ifthenelse{\equal{#1}{'}}
			{\ensuremath{\sideset{}{^{\prime}}{\textstyle\sum}_{\substack{#2 \\ #3}}}}
			{\ensuremath{\sideset{}{^{#1}}{\textstyle\sum}_{\substack{#2 \\ #3}}}}
		}
	}
}
\def\sumsstxt@iii[#1]#2#3[#4]{
	\@ifnextchar[
	{\sumsstxt@iv[#1]{#2}{#3}{#4}}
	{
		\ifthenelse{\isempty{#1}}
		{\ensuremath{\textstyle\sum_{\substack{#2 \\ #3 \\ #4}}}}
		{
			\ifthenelse{\equal{#1}{'}}
			{\ensuremath{\sideset{}{^{\prime}}{\textstyle\sum}_{\substack{#2 \\ #3 \\ #4}}}}
			{\ensuremath{\sideset{}{^{#1}}{\textstyle\sum}_{\substack{#2 \\ #3 \\ #4}}}}
		}
	}
}
\def\sumsstxt@iv[#1]#2#3#4[#5]{
	{\ifthenelse{\isempty{#1}}
		{\ensuremath{\textstyle\sum_{\substack{#2 \\ #3 \\ #4 \\ #5}}}}
		{
			\ifthenelse{\equal{#1}{'}}
			{\ensuremath{\sideset{}{^{\prime}}{\textstyle\sum}_{\substack{#2 \\ #3 \\ #4 \\ #5}}}}
			{\ensuremath{\sideset{}{^{#1}}{\textstyle\sum}_{\substack{#2 \\ #3 \\ #4 \\ #5}}}}
		}
	}
}

\def\prods{
	\@ifnextchar[
	{\prods@i}
	{\ensuremath{\prod}}
}
\def\prods@i[#1]{
	\@ifnextchar[
	{\prods@ii{#1}}
	{\ensuremath{\prod_{#1}}}
}
\def\prods@ii#1[#2]{
	\@ifnextchar[
	{\prods@iii{#1}{#2}}
	{\ensuremath{\prod_{\substack{#1 \\ #2}}}}
}
\def\prods@iii#1#2[#3]{
	\@ifnextchar[
	{\prods@iv{#1}{#2}{#3}}
	{\ensuremath{\prod_{\substack{#1 \\ #2 \\ #3}}}}
}
\def\prods@iv#1#2#3[#4]{
	{\ensuremath{\prod_{\substack{#1 \\ #2 \\ #3 \\ #4}}}}
}
\newcommand{\RNum}[1]{\uppercase\expandafter{\romannumeral #1\relax}}

\title[A generalization of  Piatetski--Shapiro sequences (II)]
      {A generalization of  Piatetski--Shapiro sequences (II)}

\author[Jinjiang Li]{Jinjiang Li}

\address{Department of Mathematics, China University of Mining and Technology, Beijing, 100083, People's Republic of China}

\email{jinjiang.li.math@gmail.com}

\author[Jinyun Qi]{Jinyun Qi}

\address{Department of Mathematics, Northwest University, Xi'an, Shaanxi, 710127, People's Republic of China}

\email{xjqi@stumail.nwu.edu.cn}

\author[Min Zhang]{Min Zhang}

\address{School of Applied Science, Beijing Information Science and Technology University, Beijing, 100192,
         People's Republic of China}

\email{min.zhang.math@gmail.com}

\begin{document}

\footnotetext[1]{Jinyun Qi is the corresponding author.}

\begin{abstract}
Suppose that $\alpha,\beta\in\mathbb{R}$. Let $\alpha\geqslant1$ and $c$ be a real number in the range
$1<c< 12/11$. In this paper, it is proved that there exist infinitely many primes in the generalized Piatetski--Shapiro sequence, which is defined by $(\lfloor\alpha n^c+\beta\rfloor)_{n=1}^\infty$. Moreover, we also prove that there exist infinitely many Carmichael numbers composed entirely of primes from the generalized Piatetski--Shapiro sequences with $c\in(1,\frac{19137}{18746})$. The two theorems constitute improvements upon previous results by Guo and Qi \cite{Guo-Qi-2021}.
\end{abstract}

\maketitle


{\textbf{MR(2020) Subject Classification}}: 11N05, 11L07, 11N80, 11B83



\section{Introduction}

For $1<c\not\in\mathbb{N}$, the Piatetski--Shapiro sequences are sequences of the form
\begin{equation*}
\mathscr{N}^{(c)}:=(\lfloor{n^c}\rfloor)_{n=1}^\infty.
\end{equation*}
For fixed real numbers $\alpha$ and $\beta$, the associated non--homogeneous Beatty sequence is the sequence of integers defined by
\begin{equation*}
\mathcal{B}_{\alpha,\beta}:=\big(\lfloor\alpha n+\beta\rfloor\big)_{n=1}^\infty,
\end{equation*}
where $\lfloor t\rfloor$ denotes the integral part of any $t\in\mathbb{R}$. Such sequences are also called generalized arithmetic progressions.
 Let $\alpha\geqslant1$ and $\beta$ be real numbers. Denote $\mathscr{N}_{\alpha,\beta}^{(c)}$ by the generalized Piatetski--Shapiro sequences
\begin{equation*}
\mathscr{N}_{\alpha,\beta}^{(c)}=\big(\lfloor\alpha n^c+\beta\rfloor\big)_{n=1}^{\infty}.
\end{equation*}
Note that the special case $\mathscr{N}_{1,0}^{(c)}$ is the classical Piatetski--Shapiro sequences. Let
\begin{equation*}
\pi(x;q,a):=\#\big\{p\leqslant x:p\equiv a\pmod q\big\}
\end{equation*}
and
\begin{equation*}
\pi_{\alpha,\beta,c}(x;q,a):=\#\big\{p\leqslant x: p \in \mathscr{N}_{\alpha,\beta}^{(c)}
\,\,\text{and}\,\, p\equiv a\pmod q\big\}.
\end{equation*}
Recently, Guo and Qi \cite{Guo-Qi-2021} gave an asymptotic formula of $\pi_{\alpha,\beta,c}(x;q,a)$ for $1<c<\frac{14}{13}$. Moreover, they also proved that there exist infinitely many Carmichael numbers composed entirely of primes from generalized Piatetski--Shapiro sequences with $1<c<\frac{64}{63}$.
In this paper, we shall continue to improve the range of $c$ in these problems and establish the following two results by improving the estimate of the weighted exponential sum
\begin{equation*}
	\sum_{1\leqslant h\leqslant H}\Bigg|\sum_{x/2<n\leqslant x}\Lambda(n)\textbf{e}(\theta hn^\gamma + \Xi n)\Bigg|,
\end{equation*}
where $H,\theta,\gamma, \Xi$ are positive numbers satisfying $H\geqslant1$ and $0<\theta, \gamma, \Xi<1$.

\begin{theorem}\label{thm:main}
Let $a$ and $q$ be coprime integers with $q\geqslant1$. For fixed $1<c<\frac{12}{11}$ and $\gamma=c^{-1}$, we have
\begin{align}\label{Thm-1-asymp}
        \pi_{\alpha,\beta,c}(x;q,a)
= &\,\, \alpha^{-\gamma}\gamma x^{\gamma-1}\pi(x;q,a)
                 \nonumber \\
  &\,\, + \alpha^{-\gamma}\gamma(1-\gamma)\int_2^xu^{\gamma-2}\pi(u;q,a)\mathrm{d}u
          +O\big(x^{7\gamma/13+11/26+\varepsilon}\big).
\end{align}
\end{theorem}

\noindent
Moreover, define
\begin{equation*}
\pi_{\alpha,\beta,c}(x):=\pi_{\alpha,\beta,c}(x;1,1)=\#\big\{p\leqslant x: p\in\mathscr{N}_{\alpha,\beta}^{(c)}\big\}.
\end{equation*}
Then we conclude that
\begin{corollary}
Suppose that $\alpha \ge 1$ and $\beta$ are real numbers. Then for $1<c<\frac{12}{11}$, there holds
\begin{equation}\label{Coro-1}
\pi_{\alpha,\beta,c}(x)=\frac{x^\gamma}{\alpha^\gamma\log x}+O\bigg(\frac{x^\gamma}{\log^2x}\bigg).
\end{equation}
\end{corollary}

In the end, we improve the theorem in~\cite{Guo-Qi-2021} related to Carmichael numbers, which are the composite natural numbers $N$ with the property that $N | (a^N - a)$ for every integer $a$.

\begin{theorem}\label{thm:2}
For every $c \in (1, \frac{19137}{18746})$, there are infinitely many Carmichael numbers composed entirely of the primes from the set $\mathscr{N}_{\alpha,\beta}^{(c)}$.
\end{theorem}

\section{Preliminaries}

\subsection{Notation}

We denote by $\lfloor t\rfloor$ and $\{t\}$ the integral part and the fractional part of $t$, respectively.
As usual, we put
\begin{equation*}
\textbf{e}(t):= e^{2\pi it}.
\end{equation*}
Throughout this paper, we make considerable use of the sawtooth function, which is defined by
\begin{equation*}
\psi(t):=t-\lfloor t\rfloor-\frac{1}{2}=\{t\}-\frac{1}{2}.
\end{equation*}
The letter $p$ always denotes a prime. For the generalized Piatetski--Shapiro sequence
$(\lfloor\alpha n^c+\beta\rfloor)_{n=1}^\infty$, we denote $\gamma:=c^{-1}$ and $\theta:=\alpha^{-\gamma}$.
We use the notation of the form $m\sim M$ as an abbreviation for $M<m\leqslant 2M$.

Throughout the paper, implied constants in symbols $O$, $\ll$ and $\gg$ may depend (where obvious) on the parameters $\alpha,\beta, c, \varepsilon$ but are absolute otherwise. For given functions $F$ and $G$, the notations $F\ll G$, $G\gg F$ and $F=O(G)$ are all equivalent to the statement that the inequality $|F|\leqslant\mathcal{C}|G|$ holds with some constant $\mathcal{C}>0$.

\subsection{Technical lemmas}

We need the following well--known approximation of Vaaler~\cite{Vaal}.

\begin{lemma}\label{lem:Vaaler}
	For any $H\geqslant 1$, there exist numbers $a_h,b_h$ such that
\begin{equation*}
  \bigg|\psi(t)-\sum_{0<|h|\leqslant H}a_h\,\mathbf{e}(th)\bigg|
  \leqslant\sum_{|h|\leqslant H}b_h\,\mathbf{e}(th),\qquad
  a_h\ll\frac{1}{|h|},\qquad b_h\ll\frac{1}{H}.
\end{equation*}
\end{lemma}

\begin{lemma}\label{Heath-Brown-identity}
	Let $z\geqslant1$ and $k\geqslant1$. Then, for any $n\leqslant2z^k$, there holds
	\begin{equation*}
		\Lambda(n)=\sum_{j=1}^k(-1)^{j-1}\binom{k}{j}\mathop{\sum\cdots\sum}_{\substack{n_1n_2\cdots n_{2j}=n\\
				n_{j+1},\dots,n_{2j}\leqslant z }}(\log n_1)\mu(n_{j+1})\cdots\mu(n_{2j}).
	\end{equation*}
\end{lemma}
\begin{proof}
	See the arguments on pp. 1366--1367 of Heath--Brown \cite{HB}.
\end{proof}

\begin{lemma}\label{compare}
	Suppose that
	\begin{equation*}
		L(H)=\sum_{i=1}^mA_iH^{a_i}+\sum_{j=1}^nB_jH^{-b_j},
	\end{equation*}
	where~$A_i,\,B_j,\,a_i$~and~$b_j$~are positive. Assume further that $H_1\leqslant H_2$. Then there exists
	some $\mathscr{H}$ with
	$H_1\leqslant\mathscr{H}\leqslant H_2$ and
	\begin{equation*}
		L(\mathscr{H})\ll \sum_{i=1}^mA_iH_1^{a_i}+\sum_{j=1}^nB_jH_2^{-b_j}+\sum_{i=1}^m\sum_{j=1}^n\big(A_i^{b_j}B_j^{a_i}\big)^{1/(a_i+b_j)}.
	\end{equation*}
	The implied constant depends only on $m$ and $n$.
\end{lemma}
\begin{proof}
	See Lemma 3 of Srinivasan \cite{Srin}.
\end{proof}

For real numbers $\theta,\Xi\in[0,1]$, the sum of the form
\begin{equation*}
	\sum_{0<|h|\leqslant H}\Bigg|\mathop{\sum_{k\sim K}\sum_{\ell\sim L}}_{KL\asymp x}a_kb_\ell
    \mathbf{e} \big(\theta h(k\ell)^\gamma + \Xi k \ell\big)\Bigg|
\end{equation*}
with $|a_k|\ll x^\varepsilon,|b_\ell|\ll x^\varepsilon$ for every fixed $\varepsilon>0$, it is usually called a ``Type I'' sum, denoted by $S_I(K,L)$, if $b_\ell=1$ or $b_\ell=\log\ell$; otherwise it is called a ``Type II'' sum, denoted by $S_{II}(K,L)$.

\begin{lemma}\label{derivative-estimate}
	Suppose that $f(x):[a, b]\to \mathbb{R}$ has continuous derivatives of arbitrary
	order on $[a,b]$, where $1\leqslant a<b\leqslant2a$. Suppose further that
	\begin{equation*}
		\big|f^{(j)}(x)\big|\asymp \lambda_j,\qquad j\geqslant1, \qquad x\in[a, b].
	\end{equation*}
	Then we have
	\begin{equation}\label{2nd-deri-estimate}
		\sum_{a<n\leqslant b}e\big(f(n)\big)\ll a\lambda_2^{1/2}+\lambda_2^{-1/2},
	\end{equation}
	and
	\begin{equation}\label{3rd-deri-estimate}
		\sum_{a<n\leqslant b}e\big(f(n)\big)\ll a\lambda_3^{1/6}+\lambda_3^{-1/3}.
	\end{equation}
\end{lemma}
\begin{proof}
	For (\ref{2nd-deri-estimate}), one can see Corollary 8.13 of Iwaniec and Kowalski \cite{IwKo}, or Theorem 5
	of Chapter 1 in Karatsuba \cite{Kara}. For (\ref{3rd-deri-estimate}), one can see Corollary 4.2 of
	Sargos \cite{Sarg}.
\end{proof}

\begin{lemma}\label{Type-I-sum}
  Suppose that $|a_k|\ll 1,b_\ell=1$ or $\log\ell,KL\asymp x$. Then if $K\ll x^{1/2}$, there holds
\begin{equation*}
   S_{I}(K,L)\ll H^{7/6}x^{\gamma/6+3/4}+H^{2/3}x^{1-\gamma/3}.
\end{equation*}
\end{lemma}
\begin{proof}
    Set $f(\ell)=\theta h(k\ell)^\gamma+\Xi k\ell$. It is easy to see that
\begin{equation*}
  f'''(\ell)=\gamma(\gamma-1)(\gamma-2)\theta hk^\gamma\ell^{\gamma-3}\asymp |h|K^{\gamma}L^{\gamma-3}.
\end{equation*}
If $K\ll x^{1/2}$, then by (\ref{3rd-deri-estimate}) of Lemma \ref{derivative-estimate}, we deduce that
\begin{align*}
            x^{-\varepsilon}\cdot S_{I}(K,L)
  \ll & \,\, \sum_{0<|h|\leqslant H}\sum_{k\sim K}
             \Bigg|\sum_{\ell\sim L}\mathbf{e}\big(f(\ell)\big)\Bigg|
                 \nonumber \\
  \ll & \,\, \sum_{0<|h|\leqslant H}\sum_{k\sim K}
             \Big(L\big(|h|K^{\gamma}L^{\gamma-3}\big)^{1/6}+\big(|h|K^{\gamma}L^{\gamma-3}\big)^{-1/3}\Big)
                 \nonumber \\
  \ll & \,\, \sum_{0<|h|\leqslant H}
             \Big(|h|^{1/6}x^{\gamma/6+1/2}K^{1/2}+|h|^{-1/3}x^{1-\gamma/3}\Big)
                 \nonumber \\
   \ll & \,\, H^{7/6}x^{\gamma/6+3/4}+H^{2/3}x^{1-\gamma/3},
\end{align*}
which completes the proof of Lemma \ref{Type-I-sum}.
\end{proof}

\begin{lemma}\label{Type-II-sum}
  Suppose that $|a_k|\ll 1,|b_\ell|\ll1$ with $k\sim K,\ell\sim L$ and $KL\asymp x$. Then if $x^{1/2}\ll K\ll x^{19/25}$, there holds
\begin{equation*}
  S_{II}(K,L)\ll H^{5/4}x^{\gamma/4+5/8}+H^{3/4}x^{1-\gamma/4}+Hx^{22/25}+H^{7/6}x^{\gamma/6+3/4}.
\end{equation*}
\end{lemma}
\begin{proof}
  Let $Q$, which satisfies $1<Q<L$, be a parameter which will be chosen later. By the Weyl--van der Corput inequality
  (see Lemma 2.5 of Graham and Kolesnik \cite{GraKol}), we have
\begin{equation}\label{Weyl-inequality}
  \Bigg|\mathop{\sum_{k\sim K}\sum_{\ell\sim L}}_{KL\asymp x}a_kb_\ell \mathbf{e}
  \big(\theta h(k\ell)^\gamma+\Xi k\ell\big)\Bigg|^2
  \ll K^2L^2Q^{-1}+KLQ^{-1}\sum_{\ell\sim L}\sum_{0<|q|\leqslant Q}\big|\mathfrak{S}(q;\ell)\big|,
\end{equation}
where
\begin{equation*}
  \mathfrak{S}(q;\ell)=\sum_{k\in\mathcal{I}(q;\ell)}\mathbf{e}\big(g(k)\big)
\end{equation*}
with
\begin{equation*}
  g(k)=\theta hk^\gamma\big(\ell^\gamma-(\ell+q)^\gamma\big)-\Xi kq.
\end{equation*}
It is easy to see that
\begin{equation*}
  g''(k)=\gamma(\gamma-1)\theta hk^{\gamma-2}\big(\ell^\gamma-(\ell+q)^\gamma\big)\asymp |h|K^{\gamma-2}L^{\gamma-1}|q|.
\end{equation*}
By (\ref{2nd-deri-estimate}) of Lemma \ref{derivative-estimate}, we have
\begin{equation}\label{Type-II-inner}
  \mathfrak{S}(q;\ell)\ll K\big(|h|K^{\gamma-2}L^{\gamma-1}|q|\big)^{1/2}+\big(|h|K^{\gamma-2}L^{\gamma-1}|q|\big)^{-1/2}.
\end{equation}
Putting (\ref{Type-II-inner}) into (\ref{Weyl-inequality}), we derive that
\begin{align*}
     & \,\, \Bigg|\mathop{\sum_{k\sim K}\sum_{\ell\sim L}}_{KL\asymp x}a_kb_\ell \mathbf{e}
            \big(\theta h(k\ell)^\gamma+\Xi k\ell\big)\Bigg|^2
                \nonumber \\
 \ll & \,\, K^2L^2Q^{-1}+KLQ^{-1}
                \nonumber \\
   & \,\,\times\sum_{\ell\sim L}\sum_{0<|q|\leqslant Q}
            \big(|h|^{1/2}K^{\gamma/2}L^{\gamma/2-1/2}|q|^{1/2}+|h|^{-1/2}K^{1-\gamma/2}L^{1/2-\gamma/2}|q|^{-1/2}\big)
                \nonumber \\
 \ll & \,\,K^2L^2Q^{-1}+KLQ^{-1}\big(|h|^{1/2}K^{\gamma/2}L^{\gamma/2+1/2}Q^{3/2}
           +|h|^{-1/2}K^{1-\gamma/2}L^{3/2-\gamma/2}Q^{1/2}\big)
                \nonumber \\
 \ll & \,\, K^2L^2Q^{-1}+|h|^{1/2}K^{1+\gamma/2}L^{\gamma/2+3/2}Q^{1/2}+|h|^{-1/2}K^{2-\gamma/2}L^{5/2-\gamma/2}Q^{-1/2}.
\end{align*}
By noting that $1\leqslant Q\leqslant L$, it follows from Lemma \ref{compare} that there exists an optimal $Q$ such that
\begin{align*}
   & \,\, \Bigg|\mathop{\sum_{k\sim K}\sum_{\ell\sim L}}_{KL\asymp x}a_kb_\ell \mathbf{e}
          \big(\theta h(k\ell)^\gamma+\Xi k\ell\big)\Bigg|^2
                \nonumber \\
\ll & \,\, |h|^{1/2}x^{\gamma/2+3/2}K^{-1/2}+Kx+|h|^{-1/2}x^{2-\gamma/2}+|h|^{1/3}x^{\gamma/3+5/3}K^{-1/3}+K^{-1/2}x^2,
\end{align*}
which implies
\begin{align*}
           \Bigg|\mathop{\sum_{k\sim K}\sum_{\ell\sim L}}_{KL\asymp x}a_kb_\ell \mathbf{e}
           \big(\theta h(k\ell)^\gamma+\Xi k\ell\big)\Bigg|
\ll & \,\, |h|^{1/4}x^{\gamma/4+3/4}K^{-1/4}+|h|^{-1/4}x^{1-\gamma/4}
                \nonumber \\
    & \,\, +K^{1/2}x^{1/2}+|h|^{1/6}x^{\gamma/6+5/6}K^{-1/6}+K^{-1/4}x.
\end{align*}
Therefore, from the above estimate and the condition $x^{1/2}\ll K\ll x^{19/25}$, we obtain
\begin{align*}
           \big| S_{II}(K,L) \big|
\ll & \,\, \sum_{0<|h|\leqslant H}\Big(|h|^{1/4}x^{\gamma/4+3/4}K^{-1/4}
             +|h|^{-1/4}x^{1-\gamma/4}
               \nonumber \\
  & \,\,   \qquad +K^{1/2}x^{1/2}+|h|^{1/6}x^{\gamma/6+5/6}K^{-1/6}+K^{-1/4}x\Big)
               \nonumber \\
\ll & \,\,  H^{5/4} x^{\gamma/4+3/4}K^{-1/4}+H^{3/4}x^{1-\gamma/4}+HK^{1/2}x^{1/2}
               \nonumber \\
    & \,\,  \qquad +H^{7/6}x^{\gamma/6+5/6}K^{-1/6}+HK^{-1/4}x
               \nonumber \\
\ll & \,\,  H^{5/4}x^{\gamma/4+5/8}+H^{3/4}x^{1-\gamma/4}+Hx^{22/25}+H^{7/6}x^{\gamma/6+3/4},
\end{align*}
which completes the proof of Lemma \ref{Type-II-sum}.
\end{proof}

\section{Proof of Theorem~\ref{thm:main}}

By a same argument of \cite[Section~3]{Guo-Qi-2021}, we have
\begin{align}\label{pi-decom-z}
  \pi_{\alpha, \beta, c}(x;q,a)
= \Sigma_1(x)+\Sigma_2(x)+O(x^{\gamma-1}),
\end{align}
where
\begin{equation*}
\Sigma_1(x)=\theta\gamma\sum_{\substack{p\leqslant x\\ p\equiv a\pmod q}}p^{\gamma-1},
\end{equation*}
and
\begin{equation*}
\Sigma_2(x)=\sum_{\substack{p\leqslant x\\ p\equiv a\pmod q}}\Big(\psi\big(-\theta(p+1-\beta)^\gamma\big)
         -\psi\big(-\theta(p-\beta)^\gamma\big)\Big).
\end{equation*}
For $\Sigma_1(x)$, by partial summation, we get
\begin{align}\label{Sigma_1-asymp}
          \Sigma_1(x)
 = & \,\, \theta\gamma\int_2^xu^{\gamma-1}\mathrm{d}\Bigg(\sum_{\substack{p\leqslant u\\ p\equiv a\pmod q}}1\Bigg)
          =\theta\gamma\int_2^xu^{\gamma-1}\mathrm{d}\pi(u;q,a)
                 \nonumber \\
 = & \theta\gamma x^{\gamma-1}\pi(x;q,a)-\theta\gamma(\gamma-1)\int_2^xu^{\gamma-2}\pi(u;q,a)\mathrm{d}u.
\end{align}
Next, we turn our attention to $\Sigma_2(x)$. Define
\begin{equation*}
  \mathcal{H}(x)=\sum_{\substack{n\leqslant x\\ n\equiv a\,(\!\!\bmod q)}}
  \Lambda(n)\big(\psi(-\theta(n+1-\beta)^\gamma)-\psi(-\theta(n-\beta)^\gamma)\big),
\end{equation*}
\begin{equation*}
  \mathcal{J}(x)=\sum_{\substack{p\leqslant x\\ p\equiv a\,(\!\!\bmod q)}}
  (\log p)\big(\psi(-\theta(p+1-\beta)^\gamma)-\psi(-\theta(p-\beta)^\gamma)\big).
\end{equation*}
Trivially, we have
\begin{equation}\label{H-J-eq}
  \mathcal{H}(x)=\mathcal{J}(x)+O(x^{1/2}),
\end{equation}
Moreover, it follows from partial summation that
\begin{equation}\label{Sigma-2-decom}
  \Sigma_2(x)=\int_2^x\frac{1}{\log u}\mathrm{d}\mathcal{J}(u)=\frac{\mathcal{J}(x)}{\log x}
  +\int_2^x\frac{\mathcal{J}(u)}{u\log^2u}\mathrm{d}u.
\end{equation}
In order to obtain the upper bound estimate of $\Sigma_2(x)$, it follows from (\ref{H-J-eq}) and (\ref{Sigma-2-decom}) that we only need to derive the upper bound estimate of $\mathcal{H}(x)$. By a splitting argument, it is sufficient to give the upper bound estimate of the following sum
\begin{equation*}
 \mathcal{S}:=\sum_{\substack{x/2<n\leqslant x\\ n\equiv a\pmod q}} \Lambda(n)\Big(\psi\big(-\theta(n+1-\beta)^\gamma\big)-\psi\big(-\theta(n-\beta)^\gamma\big)\Big).
\end{equation*}
According to Vaaler's approximation, i.e. Lemma \ref{lem:Vaaler}, we can write
\begin{equation}\label{S-decompo}
\mathcal{S}=\mathcal{S}_1+O(|\mathcal{S}_2|),
\end{equation}
where
\begin{equation*}
\mathcal{S}_1=\sum_{\substack{x/2<n\leqslant x\\ n\equiv a\pmod q}} \Lambda(n)
\sum_{0<|h|\leqslant H}a_h\big(\mathbf{e}(\theta h(n+1-\beta)^\gamma)-\mathbf{e}(\theta h(n-\beta)^\gamma)\big),
\end{equation*}
\begin{equation*}
\mathcal{S}_2=\sum_{\substack{x/2<n\leqslant x\\ n\equiv a\pmod q}} \Lambda(n)
\sum_{|h|\leqslant H}b_h\big(\mathbf{e}(\theta h(n+1-\beta)^\gamma)+\mathbf{e}(\theta h(n-\beta)^\gamma)\big).
\end{equation*}
Moreover, we split $\mathcal{S}_1$ into two parts
\begin{equation}\label{S1-two-parts}
   \mathcal{S}_1=\mathcal{S}_1^{(1)}+\mathcal{S}_1^{(2)},
\end{equation}
where
\begin{equation*}
\mathcal{S}_1^{(1)}=\sum_{\substack{x/2<n\leqslant x\\ n\equiv a\pmod q}} \Lambda(n)
\sum_{0<|h|\leqslant H}a_h\big(\mathbf{e}(\theta h(n+1-\beta)^\gamma)-\mathbf{e}(\theta hn^\gamma)\big),
\end{equation*}
\begin{equation*}
\mathcal{S}_1^{(2)}=\sum_{\substack{x/2<n\leqslant x\\ n\equiv a\pmod q}} \Lambda(n)
\sum_{0<|h|\leqslant H}a_h\big(\mathbf{e}(\theta hn^\gamma)-\mathbf{e}(\theta h(n-\beta)^\gamma)\big).
\end{equation*}
Firstly, we shall consider the upper bound of $\mathcal{S}_1^{(1)}$. Let
\begin{equation*}
 \phi_{h}(t):=\mathbf{e}\big(h((t+1-\beta)^\gamma-t^\gamma)\big)-1.
\end{equation*}
Therefore, $\mathcal{S}_1^{(1)}$ is
\begin{equation*}
 =\sum_{\substack{x/2<n\leqslant x\\ n\equiv a\pmod q}}\Lambda(n)\sum_{0<|h|\leqslant H}a_h
  \phi_{h}(n)\mathbf{e}(\theta hn^\gamma)
 =\sum_{0<|h|\leqslant H}a_h\sum_{\substack{x/2<n\leqslant x\\ n\equiv a\pmod q}} \Lambda(n)
  \phi_{h}(n)\mathbf{e}(\theta hn^\gamma),
\end{equation*}
which combined with the upper bound $a_{h}\ll|h|^{-1}$ yields
\begin{equation*}
    \mathcal{S}_1^{(1)}\ll\sum_{0<|h|\leqslant H}\frac{1}{|h|}\Bigg|
    \sum_{\substack{x/2<n\leqslant x\\ n\equiv a\pmod q}} \Lambda(n)\phi_{h}(n)\mathbf{e}(\theta hn^\gamma)\Bigg|.
\end{equation*}
It follows from partial summation and the bounds
\begin{equation*}
\phi_{h}(t)\ll|h|t^{\gamma -1}\qquad \textrm{and} \qquad\frac{\partial\phi_{h}(t)}{\partial t}\ll|h|t^{\gamma-2}
\end{equation*}
that
\begin{align}\label{S_1-1-upper}
           \mathcal{S}_1^{(1)}
\ll & \,\, \sum_{0<|h|\leqslant H}\frac{1}{|h|}\Bigg|
           \int_{\frac{x}{2}}^x\phi_{h}(t)\mathrm{d}\Bigg(
           \sum_{\substack{x/2<n\leqslant t\\ n\equiv a\pmod q}}\Lambda(n)\mathbf{e}(\theta hn^\gamma)\Bigg)\Bigg|
                   \nonumber \\
\ll & \,\, \sum_{0<|h|\leqslant H}\frac{1}{|h|}\Big|\phi_{h}(x)\Big|\Bigg|
           \sum_{\substack{x/2<n\leqslant x\\ n\equiv a\pmod q}}\Lambda(n)\mathbf{e}(\theta hn^\gamma)\Bigg|
                   \nonumber \\
    & \,\, +\sum_{0<|h|\leqslant H}\frac{1}{|h|}\int_{\frac{x}{2}}^x
           \Bigg|\frac{\partial\phi_{h}(t)}{\partial t}\Bigg|
           \Bigg|\sum_{\substack{x/2<n\leqslant t\\ n\equiv a\pmod q}}\Lambda(n)\mathbf{e}(\theta hn^\gamma)\Bigg|\mathrm{d}t
                   \nonumber \\
\ll & \,\  x^{\gamma-1}\times\max_{x/2<t\leqslant x}\sum_{0<|h|\leqslant H}
           \Bigg|\sum_{\substack{x/2<n\leqslant t\\ n\equiv a\pmod q}}\Lambda(n)\mathbf{e}(\theta hn^\gamma)\Bigg|.
\end{align}
For $\mathcal{S}_1^{(2)}$, by a similar argument with $\phi_h(t)$ replaced by $\Xi_h(t)$ defined by
\begin{equation*}
  \Xi_h(t)=1-\mathbf{e}\big(\theta h((t-\beta)^\gamma-t^\gamma)\big),
\end{equation*}
which satisfies
\begin{equation*}
\Xi_{h}(t)\ll|h|t^{\gamma -1}\qquad \textrm{and} \qquad\frac{\partial\Xi_{h}(t)}{\partial t}\ll|h|t^{\gamma-2},
\end{equation*}
one can also derive that
\begin{equation}\label{S_1-2-upper}
\mathcal{S}_1^{(2)}\ll x^{\gamma-1}\times\max_{x/2<t\leqslant x}\sum_{0<|h|\leqslant H}
           \Bigg|\sum_{\substack{x/2<n\leqslant t\\ n\equiv a\pmod q}}\Lambda(n)\mathbf{e}(\theta hn^\gamma)\Bigg|.
\end{equation}
In order to prove Theorem \ref{thm:main}, it is sufficient to give the upper bound estimate of the following sum
\begin{equation}\label{suff-condition}
  \max_{x/2<t\leqslant x}\sum_{0<|h|\leqslant H}
  \Bigg|\sum_{\substack{x/2<n\leqslant t\\ n\equiv a\pmod q}}\Lambda(n)\mathbf{e}(\theta hn^\gamma)\Bigg|.
\end{equation}
By using the well--known orthogonality
\begin{equation*}
 \frac{1}{q}\sum_{m=1}^q\mathbf{e}\bigg(\frac{(n-a)m}{q}\bigg)=
 \begin{cases}
   1, & \textrm{if $q|n-a$}, \\
   0, & \textrm{if $q\nmid n-a$},
 \end{cases}
\end{equation*}
we can represent the innermost sum in (\ref{suff-condition}) as
\begin{equation}\label{inner-trans}
  \sum_{\substack{x/2<n\leqslant t\\ n\equiv a\pmod q}}\Lambda(n)\mathbf{e}(\theta hn^\gamma)
  =\frac{1}{q}\sum_{m=1}^q\sum_{x/2<n\leqslant t}\Lambda(n)\mathbf{e}\bigg(\theta hn^\gamma+\frac{(n-a)m}{q}\bigg).
\end{equation}
From (\ref{suff-condition}) and (\ref{inner-trans}), we know that it suffices to estimate
\begin{equation*}
  \max_{x/2<t\leqslant x}\sum_{0<|h|\leqslant H}
  \Bigg|\sum_{x/2<n\leqslant t}\Lambda(n)\mathbf{e}\big(\theta hn^\gamma+nmq^{-1}\big)\Bigg|.
\end{equation*}
By Heath--Brown's identity, i.e. Lemma \ref{Heath-Brown-identity}, with $k=3$, one can see that the exponential sum
\begin{equation*}
  \max_{x/2<t\leqslant x}\sum_{0<|h|\leqslant H}
  \Bigg|\sum_{x/2<n\leqslant t}\Lambda(n)\mathbf{e}\big(\theta hn^\gamma+nmq^{-1}\big)\Bigg|
\end{equation*}
can be written as linear combination of $O(\log^6x)$ sums, each of which is of the form
\begin{align}\label{single-sum}
          \mathcal{T}^*
:= & \,\, \sum_{0<|h|\leqslant H}\Bigg|\sum_{n_1\sim N_1}\cdots
          \sum_{n_6\sim N_6}(\log n_1)\mu(n_4)\mu(n_5)\mu(n_6)
                   \nonumber \\
   & \,\, \qquad\qquad\qquad\times \mathbf{e}\big(\theta h(n_1n_2\cdots n_6)^\gamma+(n_1n_2\cdots n_6)mq^{-1}\big)\Bigg|,
\end{align}
where $N_1N_2\cdots N_6\asymp x$; $2N_i\leqslant(2x)^{1/3},i=4,5,6$ and some $n_i$ may only take value $1$.
Therefore, it is sufficient for us to give upper bound estimate for each $\mathcal{T}^*$ defined as in (\ref{single-sum}). Next, we will consider three cases.

\noindent
\textbf{Case 1.} If there exists an $N_j$ such that $N_j\geqslant x^{1/2}$, then we must have $j\leqslant3$ for the fact that
$N_i\ll x^{1/3}$ with $i=4,5,6$. Let
\begin{equation*}
  k=\prod_{\substack{1\leqslant i\leqslant6\\ i\not=j}}n_i,\qquad \ell=n_j,\qquad
  K=\prod_{\substack{1\leqslant i\leqslant6\\ i\not=j}}N_i,\qquad L=N_j.
\end{equation*}
In this case, we can see that $\mathcal{T}^*$ is a sum of ``Type I'' satisfying $K\ll x^{1/2}$. By Lemma \ref{Type-I-sum},
we have
\begin{equation*}
  x^{-\varepsilon}\cdot \mathcal{T}^*\ll H^{7/6}x^{\gamma/6+3/4}+H^{2/3}x^{1-\gamma/3}.
\end{equation*}

\noindent
\textbf{Case 2.} If there exists an $N_j$ such that $x^{6/25}\leqslant N_j<x^{1/2}$, then we take
\begin{equation*}
  k=\prod_{\substack{1\leqslant i\leqslant6\\ i\not=j}}n_i,\qquad \ell=n_j,\qquad
  K=\prod_{\substack{1\leqslant i\leqslant6\\ i\not=j}}N_i,\qquad L=N_j.
\end{equation*}
Thus, $\mathcal{T}^*$ is a sum of ``Type II'' satisfying $x^{1/2}\ll K\ll x^{19/25}$. By Lemma \ref{Type-II-sum},
we have
\begin{equation*}
 x^{-\varepsilon}\cdot  \mathcal{T}^*\ll H^{5/4}x^{\gamma/4+5/8}+H^{3/4}x^{1-\gamma/4}+Hx^{22/25}+H^{7/6}x^{\gamma/6+3/4}.
\end{equation*}

\noindent
\textbf{Case 3.} If $N_j<x^{6/25}\,(j=1,2,3,4,5,6)$, without loss of generality, we assume that
$N_1\geqslant N_2\geqslant\cdots\geqslant N_6$. Let $r$ denote the natural number $j$ such that
\begin{equation*}
  N_1N_2\cdots N_{j-1}<x^{6/25},\qquad N_1N_2\cdots N_j\geqslant x^{6/25}.
\end{equation*}
Since $N_1<x^{6/25}$ and $N_6<x^{6/25}$, then $2\leqslant r\leqslant5$. Thus, we have
\begin{equation*}
 x^{6/25}\leqslant N_1N_2\cdots N_r=(N_1\cdots N_{r-1})\cdot N_r<x^{6/25}\cdot x^{6/25}<x^{1/2}.
\end{equation*}
Let
\begin{equation*}
 k=\prod_{i=r+1}^6n_i,\qquad \ell=\prod_{i=1}^rn_i,\qquad K=\prod_{i=r+1}^6N_i,\qquad L=\prod_{i=1}^rN_i.
\end{equation*}
 At this time, $\mathcal{T}^*$ is a sum of ``Type II'' satisfying $x^{1/2}\ll K\ll x^{19/25}$. By Lemma \ref{Type-II-sum},
we have
\begin{equation*}
  x^{-\varepsilon}\cdot \mathcal{T}^*\ll H^{5/4}x^{\gamma/4+5/8}+H^{3/4}x^{1-\gamma/4}+Hx^{22/25}+H^{7/6}x^{\gamma/6+3/4}.
\end{equation*}
Combining the above three cases, we derive that
\begin{align*}
             x^{-\varepsilon}\cdot \mathcal{T}^*
  \ll & \,\, H^{7/6}x^{\gamma/6+3/4}+H^{2/3}x^{1-\gamma/3}+H^{5/4}x^{\gamma/4+5/8}+H^{3/4}x^{1-\gamma/4}+Hx^{22/25},
\end{align*}
which combined with (\ref{S1-two-parts}), (\ref{S_1-1-upper}), and  (\ref{S_1-2-upper}) yields
\begin{align}\label{S_1-upper-fi}
           x^{-\varepsilon}\cdot  \mathcal{S}_1
 \ll & \,\, H^{5/4}x^{5\gamma/4-3/8}+H^{3/4}x^{3\gamma/4}+Hx^{\gamma-3/25}
                    +H^{7/6}x^{7\gamma/6-1/4}.
\end{align}
Now, we focus on the upper bound of $\mathcal{S}_2$. The contribution from $h=0$ is
\begin{equation}\label{S_2-zero}
  2b_0\sum_{\substack{x/2<n\leqslant x\\ n\equiv a\pmod q}} \Lambda(n)\ll\frac{b_0x}{\varphi(q)}
  \ll xH^{-1}.
\end{equation}
On the other hand, by similar arguments of $\mathcal{S}_1$ with a shift of $n$, the contribution from $h\not=0$ is
\begin{align*}
  \ll & \,\, \sum_{\substack{x/2<n\leqslant x\\ n\equiv a\pmod q}} \Lambda(n)
             \sum_{0<|h|\leqslant H}b_h\mathbf{e}(\theta hn^\gamma)
   =  \sum_{0<|h|\leqslant H}b_h\sum_{\substack{x/2<n\leqslant x\\ n\equiv a\pmod q}}
             \Lambda(n)\mathbf{e}(\theta hn^\gamma)
                   \nonumber \\
  \ll & \,\, \frac{1}{H}\sum_{0<|h|\leqslant H}\Bigg|\sum_{\substack{x/2<n\leqslant x\\ n\equiv a\pmod q}}
             \Lambda(n)\mathbf{e}(\theta hn^\gamma)\Bigg|,
\end{align*}
which can be treated as the process of (\ref{suff-condition}) to give the upper bound
\begin{align}\label{S_2-nonzero}
  \ll & \,\, H^{1/6}x^{\gamma/6+3/4}+H^{-1/3}x^{1-\gamma/3}+H^{1/4}x^{\gamma/4+5/8} +H^{-1/4}x^{1-\gamma/4}+x^{22/25}.
\end{align}
From (\ref{S-decompo}), (\ref{S_1-upper-fi}), (\ref{S_2-zero}) and (\ref{S_2-nonzero}), we obtain
\begin{align*}
            x^{-\varepsilon}\cdot \mathcal{S}
 \ll & \,\, H^{5/4}x^{5\gamma/4-3/8}+H^{3/4}x^{3\gamma/4}+Hx^{\gamma-3/25}+H^{7/6}x^{7\gamma/6-1/4}
                      +H^{1/6}x^{\gamma/6+3/4}
                 \nonumber \\
 & \,\, +H^{1/4}x^{\gamma/4+5/8}+x^{22/25} +H^{-1/3}x^{1-\gamma/3} +H^{-1/4}x^{1-\gamma/4}+xH^{-1}.
\end{align*}
Since the above upper bound holds for any real $H\geqslant1$, using Lemma \ref{compare} we deduce that
\begin{align}\label{S-fi-up}
            x^{-\varepsilon}\cdot \mathcal{S}
 \ll & \,\, x^{5\gamma/4-3/8}+x^{3\gamma/4}+x^{\gamma-3/25}+x^{7\gamma/6-1/4}+x^{\gamma/6+3/4}+x^{\gamma/4+5/8}
                \nonumber \\
     & \,\, +x^{22/25}+x^{5\gamma/9+7/18}+x^{3\gamma/7+3/7}+x^{\gamma/2+11/25}+x^{7\gamma/13+11/26}
                \nonumber \\
     & \,\, +x^{\gamma/7+11/14}+x^{\gamma/5+7/10}.
\end{align}
By noting the fact that $\pi_{\alpha,\beta,c}(x;q,a)\ll x^\gamma$, the above bound is trivial unless the exponent of each term
in the parentheses is strictly less than $\gamma$, which means that $\gamma>11/12$. Under this condition, after eliminating lower order terms, the previous bound of $\mathcal{S}$ in (\ref{S-fi-up}) simplifies to
\begin{equation*}
   \mathcal{S}\ll x^{7\gamma/13+11/26+\varepsilon}
\end{equation*}
for any $\varepsilon>0$. Therefore, we obtain $\mathcal{H}(x)\ll x^{7\gamma/13+11/26+\varepsilon}\ll x^{\gamma-\varepsilon}$
when $\gamma>11/12$, and thus $\mathcal{J}(x)\ll x^{7\gamma/13+11/26+\varepsilon}$. Moreover, from (\ref{Sigma-2-decom})
we derive that
\begin{equation}\label{Sigma-2-upper-fi}
   \Sigma_2(x)\ll x^{7\gamma/13+11/26+\varepsilon}\ll x^{\gamma-\varepsilon}.
\end{equation}
Consequently, according to (\ref{pi-decom-z}), (\ref{Sigma_1-asymp}) and (\ref{Sigma-2-upper-fi}), we derive the asymptotic
formula (\ref{Thm-1-asymp}). This completes the proof of Theorem \ref{thm:main}.

\section{Sketch of proof of Theorem~\ref{thm:2}}
By exactly the same argument of~\cite[Section~4]{Guo-Qi-2021}, we conclude that:
\begin{theorem}
\label{thm:3}
Let $\alpha\geqslant1$ and $\beta$ be real numbers. Suppose that $c\in(1,\frac{12}{11})$. Then we have
\begin{align*}
          \vartheta_{\alpha,\beta,c}(x;q,a)
 = & \,\, \alpha^{-\gamma}\gamma x^{\gamma-1}\vartheta(x;q,a)  \\
   & \,\, +\alpha^{-\gamma}\gamma(1-\gamma)\int_2^xu^{\gamma-2}\vartheta(u;q,a)\mathrm{d}u+O(x^{7\gamma/13+11/26+\varepsilon}),
\end{align*}
where the implied constant depends only on $\alpha,\beta,c$ and $\varepsilon$.
\end{theorem}

The proof of our Theorem~\ref{thm:2} is exactly the same as ~\cite[Section~4]{Guo-Qi-2021} by switching the conditions
$$
1<c<\frac{14}{13}  \qquad \hbox{and} \qquad -\frac{13}{35} + \frac{2\gamma}{5}
$$
into
$$
1<c<\frac{12}{11} \qquad \hbox{and} \qquad -\frac{11}{26} + \frac{6\gamma}{13}.
$$
Let $ \pi(x,y) $ be the number of those for which $p-1$ is free of prime factors exceeding $y$. Let $\mathcal{E}$ be the set of numbers $E$ in the range $0<E<1$ for which
$$
\pi(x, x^{1-E}) \ge x^{1+o(1)}
$$
as $x\rightarrow \infty$, where the function implied by $o(1)$ depends on $E$. By the same argument in~\cite[Section~4]{Guo-Qi-2021}, we conclude the following statement.

\begin{lemma}
\label{lem:fin}
Let $\alpha \ge 1$ and $\beta$ be real numbers. Suppose that $c\in\big(1,\frac{38}{37}\big)$.
Let $B,B_1$ be positive real numbers such that $B_1<B<-\frac{11}{26}+\frac{6\gamma}{13}$. For any $E\in\mathcal{E}$,
there exsits a number $x_3$ depending on $c,B,B_1,E$ and $\varepsilon$, such that for any $x\geqslant x_1$ there exist at least $x^{EB+(1-B+B_1)(\gamma-1)-\varepsilon}$ Carmichael numbers up to $x$ composed solely of primes from
$\mathscr{N}_{\alpha, \beta}^{(c)}$.
\end{lemma}
Taking $B$ and $B_1$ arbitrarily close to $-\frac{11}{26}+\frac{6\gamma}{13}$, Lemma \ref{lem:fin} implies that there are infinitely many Carmichael numbers composed entirely of the primes from $\mathscr{N}_{\alpha, \beta}^{(c)}$ with
\begin{equation*}
\bigg(-\frac{11}{26}+\frac{6\gamma}{13}\bigg) E + \gamma -1 > 0.
\end{equation*}
Taking $E= 0.7039$ from \cite{BaHa2}, we eventually have $\gamma > \frac{18746}{19137}$.

\section{Acknowledgement}
The authors would like to appreciate the referee for his/her patience in refereeing this paper.
This work is supported by the National Natural Science Foundation of China (Grants No. 11901566, 12001047, 11971476, 12071238, 11971381, 11701447, 11871317, 11971382), the Fundamental Research Funds for the Central Universities (Grant No. 2021YQLX02), the National Training Program of Innovation and Entrepreneurship for Undergraduates (Grant No. 202107010), the Undergraduate Education and Teaching Reform and Research Project for China University of Mining and Technology (Beijing) (Grant No. J210703), and the Scientific Research Funds of Beijing Information Science and Technology University (Grant No. 2025035).

\end{document}